\documentclass[11pt]{amsart}
\usepackage{amsmath, amssymb}
\usepackage{amsfonts}
\usepackage{graphicx}
\usepackage{mathrsfs}
\usepackage[arrow,matrix,curve,cmtip,ps]{xy}
\usepackage{amsthm}
\usepackage{enumerate}
\usepackage{color}
\usepackage[colorlinks]{hyperref}
\usepackage{tikz}

\allowdisplaybreaks

\newtheorem{thm}{Theorem}[section]
\newtheorem{lem}[thm]{Lemma}
\newtheorem{prop}[thm]{Proposition}
\newtheorem{cor}[thm]{Corollary}
\newtheorem*{theorem*}{Theorem}
\theoremstyle{remark}
\newtheorem{rem}[thm]{Remark}

\newtheorem{defn}[thm]{Definition}
\newtheorem{ex}[thm]{Example}

\numberwithin{equation}{section}

\newcommand{\N}{\mathbb{N}}
\newcommand{\R}{\mathbb{R}}
\newcommand{\C}{\mathbb{C}}

\newcommand{\g}{\mathcal{G}_{G,E}}
\newcommand{\e}{\widetilde{E}}
\newcommand{\gr}{\mathcal{G}}
\newcommand{\Og}{\mathcal{O}_{G,E}}
\newcommand{\La}{\Lambda}
\newcommand{\la}{\lambda}
\newcommand{\Ol}{\mathcal{O}_{G,\Lambda}}

\begin{document}
\title[A dichotomy for simple self-similar graph $C^\ast$-algebras]{A dichotomy for simple self-similar graph $C^\ast$-algebras}

\author[Hossein Larki]{Hossein Larki}

\address{Department of Mathematics\\
Faculty of Mathematical Sciences and Computer\\
Shahid Chamran University of Ahvaz\\
P.O. Box: 83151-61357\\
Ahvaz\\
 Iran}
\email{h.larki@scu.ac.ir}


\date{\today}

\subjclass[2010]{46L05, 46L55}

\keywords{self-similar graph, $C^*$-algebra, pure infiniteness, stable finiteness}

\begin{abstract}

We investigate the pure infiniteness and stable finiteness of the Exel-Pardo $C^*$-algebras $\mathcal{O}_{G,E}$ for countable self-similar graphs $(G,E,\varphi)$. In particular, we associate a specific ordinary graph $\widetilde{E}$ to $(G,E,\varphi)$ such that some properties such as simpleness, stable finiteness or pure infiniteness of the graph $C^*$-algebra $C^*(\widetilde{E})$ imply that of $\mathcal{O}_{G,E}$. Among others, this follows a dichotomy for simple $\mathcal{O}_{G,E}$: if $(G,E,\varphi)$ contains no $G$-circuits, then $\mathcal{O}_{G,E}$ is stably finite; otherwise, $\mathcal{O}_{G,E}$ is purely infinite.

Furthermore, Li and Yang recently introduced self-similar $k$-graph $C^*$-algebras $\mathcal{O}_{G,\Lambda}$. We also show that when $|\Lambda^0|<\infty$ and  $\mathcal{O}_{G,\Lambda}$ is simple, then it is purely infinite.
\end{abstract}

\maketitle

\section{Introduction}

In \cite{exe17}, Exel and Pardo introduced self-similar graph $C^*$-algebras $\Og$ to give a unified framework like graph $C^*$-algebras for the Katsura's \cite{kat08} and Nekrashevych's algebras \cite{nek04,nek05}. These $C^*$-algebras were initially considered in \cite{exe17} only for countable discrete groups $G$ acting on finite graphs $E$ with no sources, and then generalized in \cite{bed17,exe18} for larger classes. Roughly speaking, Exel and Pardo attached an inverse semigroup $\mathcal{S}_{G,E}$ and the tight groupoid $\mathcal{G}_{tight} (\mathcal{S}_{G,E})$ to $(G,E,\varphi)$ such that $\Og\cong C^*(\mathcal{G}_{\mathrm{tight}}(\mathcal{S}_{G,E}))$, and then describe amenability \cite[Corollary 10.18]{exe17}, minimality \cite[Theorem 13.6]{exe17}, and effectivity (or topological principality) \cite[Corollary 14.15]{exe17} of $\mathcal{G}_{\mathrm{tight}}(\mathcal{S}_{G,E})$, and thus simplicity and pure infiniteness of $\Og$ \cite[Section 16]{exe17}, among others. Although only finite graphs are considered in \cite{exe17}, but many arguments and proofs work for countable row-finite graphs with no sources (see \cite{exe18}).

The initial aim of this note comes from a dichotomy for simple groupoid $C^*$-algebras \cite{rai18,bon18}. According to \cite[Theorem 4.7]{rai18} and \cite[Corollary 5.13]{bon18}, a simple reduced $C^*$-algebra $C^*_r(\mathcal{G})$ of ample groupoid $\mathcal{G}$ with an almost unperforated type semigroup is either purely infinite or stable finite. We explicitly describe this dichotomy for self-similar graph $C^*$-algebras $\Og$ by the underlying graphical properties. Here, we consider countable row-finite source-free graphs $E$ over an amenable (countable) group $G$ \cite{bed17,exe18}. However, our results may be generalized to any countable graph $E$ by the desingularization of \cite{exe18}.

We begin in Section 2 by reviewing necessary background on groupoid and self-similar graph $C^*$-algebras. Then, in Section 3, we generalize the Exel-Pardo's characterization of purely infinite $\Og$ to countable self-similar graphs by the groupoid approach (for not necessarily simple cases). Moreover, for certain self-similar graphs $(G,E,\varphi)$, we show that the $C^*$-algebra $\Og$ is purely infinite and simple if and only if the additive monoid of nonzero Murray-von Neumann equivalent projections in $M_\infty(\Og)$ is a group.

In Section 4, we focus on the stable finiteness of $\Og$. We attach a spacial graph $\widetilde{E}$ to $(G,E,\varphi)$ such that some properties of $\Og$- such as simplicity, pure infiniteness, and stable infiniteness- can be derived from those of the graph $C^*$-algebra $C^*(\widetilde{E})$. Then using known results about the graph $C^*$-algebras, we show that a simple $C^*$-algebra $\Og$ is stable finite if and only if the underlying $(G,E,\varphi)$ contains no $G$-circuits. In particular, we deduce a dichotomy: A simple $\Og$ is purely infinite if $(G,E,\varphi)$ has a $G$-circuit; otherwise, it is stable finite.

As the $k$-graph version of Exel-Pardo $C^*$-algebras, Li and Yang introduced self-similar $k$-graphs $(G,\La)$ and associated $C^*$-algebras $\mathcal{O}_{G,\La}$. Briefly, by a groupoid approach, they investigated their properties such as nuclearity \cite[Theorem 6.6(i)]{li18}, amenability \cite[Theorem 5.9]{li18}, and simplicity \cite[Theorem 6.6(ii)]{li18}. In Section 5, We investigate the pure infiniteness of $\mathcal{O}_{G,\La}$ for the nonsimple cases. In particular, we modify and extend \cite[Theorem 6.13]{li18}.

{\bf Acknowledgement.} The author appreciates Enrique Pardo for reviewing the initial version of the article and his helpful comments; in particular, for noting a gap in the proof of Theorem \ref{thm3.7}.


\section{Preliminaries}

\subsection{Groupoid $C^*$-algebras}

We give here a brief introduction to ample groupoids and associated $C^*$-algebras; for more details see \cite{ren80, ana00} for example. A {\it groupoid} is a small category $\gr$ with inverses. The {\it unit space of $\gr$} is the set of identity morphisms, that is $\gr^{(0)}:=\{\alpha^{-1}\alpha:\alpha\in \gr\}$. For each $\alpha\in \gr$, we may define the range $r(\alpha):=\alpha\alpha^{-1}$ and the source $s(\alpha):=\alpha^{-1}\alpha$, which satisfy $r(\alpha)\alpha=\alpha=\alpha s(\alpha)$. Hence, for $\alpha,\beta\in \gr$, the composition $\alpha\beta$ is well-defined in $\gr$ if and only if $s(\alpha)=r(\beta)$. The {\it isotropy subgroupoid of $\gr$} is defined by
$$\mathrm{Iso}(\gr):=\{\alpha\in \gr:s(\alpha)=r(\alpha)\}.$$

We work usually with groupoids $\gr$ endowed with a topology such that the maps $r,s:\gr\rightarrow \gr^{(0)}$ are continuous (in this case, $\gr$ is called a {\it topological groupoid}). A subset $B\subseteq \gr$ is called a {\it bisection} if both restrictions $r|_B$ and $s|_B$ are homeomorphisms.
We say that $\gr$ is {\it ample} in case $\gr$ has a basis of compact and open bisections.

\begin{defn}
Let $\gr$ be a topological groupoid. We say that $\gr$ is {\it effective} if the interior of $\mathrm{Iso}(\gr)$ is just $\gr^{(0)}$. Moreover, $\gr$ is called {\it topologically principal} if $\{u\in \gr^{(0)}: s^{-1}(u)\cap r^{-1}(u)=\{u\}\}$ is dense in $\gr^{(0)}$.
\end{defn}

Note that, when $\gr$ is second-countable, \cite[Proposition 3.3]{ren08} implies that $\gr$ is effective if and only if it is topologically principal. In this paper, we will work frequently with second-countable effective ample groupoids.

We now recall the definition of reduced $C^*$-algebra $C^*_r(\gr)$. Let $\gr$ be an ample groupoid. We write $C_c(\gr)$ for the complex vector space consisting of compactly supported continuous functions on $\gr$, which is an $*$-algebra with the convolution multiplication and the involution $f^*(\alpha):=\overline{f(\alpha^{-1})}$. For each unit $u\in \gr^{(0)}$ and $\gr_u:=s^{-1}(\{u\})$, let $\pi_u:C_c(\gr)\rightarrow B(\ell^2(\gr_u))$ be the left regular $*$-representation defined by
$$\pi_u(f)\delta_\alpha:=\sum_{s(\beta)=r(\alpha)}f(\beta) \delta_{\beta \alpha} \hspace{5mm} (f\in C_c(\gr),~\alpha\in \gr_u).$$
Then the {\it reduced $C^*$-algebra $C^*_r(\gr)$} is the completion of $C_c(\gr)$ under the reduced $C^*$-norm
$$\|f\|_r:=\sup_{u\in \gr^{(0)}}\|\pi_u(f)\|.$$
Moreover, there is a full $C^*$-algebra $C^*(\gr)$ associated to $\gr$, which is the completion of $C_c(\gr)$ taken over all $\|.\|_{C_c(\gr)}$-decreasing representations of $\gr$. Hence, $C^*_r(\gr)$ is a quotient of $C^*(\gr)$, and \cite[Proposition 6.1.8]{ana00} shows that they are equal if the underlying groupoid $\gr$ is amenable.

\begin{defn}[\cite{ror02}]
We say that a $C^*$-algebra $A$ is {\it purely infinite} if every nonzero hereditary $C^*$-subalgebra of $A$ contains an infinite projection.
\end{defn}

The following is analogous to \cite[Theorem 4.1]{bro15} without the minimality assumption.

\begin{prop}\label{prop2.3}
Let $\mathcal{G}$ be a second-countable Hausdorff ample groupoid and let $\mathcal{B}$ be a basis of compact open sets for $\mathcal{G}^{(0)}$. Suppose also that $\mathcal{G}$ is effective. Then $C^*_r(\mathcal{G})$ is purely infinite if and only if $1_V$ is infinite in $C^*_r(\mathcal{G})$ for every $V\in \mathcal{B}$ ($1_V$ is the characteristic function of $V$).
\end{prop}

\begin{proof}
The ``only if" implication is immediate. For the converse, suppose that every $1_V$ in $C^*_r(\mathcal{G})$ is infinite for $V\in \mathcal{B}$. Let $A$ be a nonzero hereditary $C^*$-subalgebra of $C^*_r(\mathcal{G})$ and take some positive element $0\neq a\in A$. Using the hereditary property, we may follow the proof of \cite[Proposition 5.2]{lar19} to find a projection $p\in A$ and some $V\in \mathcal{B}$ such that $p\sim 1_V$ in the Murray-von Nuemann sense. Since the infiniteness is preserved under $\sim$, then $p$ is an infinite projection, concluding the result.
\end{proof}

\subsection{Graph $C^*$-algebras}

Let $E=(E^0,E^1,r,d)$ be a directed graph with the vertex set $E^0$, the edge set $E^1$, and the range and domain maps $r,d:E^1\rightarrow E^0$. We say that $E$ is {\it row-finite} if each vertex receives at most finitely many edges. A {\it source in $E$} is a vertex $v\in E^0$ which receives no edges, i.e. $d^{-1}(v)=\emptyset$. We will write by $E^*$ the set of finite paths in $E$, that is
$$E^*:=\bigcup_{n\geq 0}E^n=\bigcup_{n\geq 0}\{\alpha=e_1\ldots e_n: e_i\in E^1, d(e_i)=r(e_{i+1})\}.$$
Then one may extend $r,d:E^*\rightarrow E^0$ by defining $r(\alpha)=r(e_1)$ and $d(\alpha)=d(e_n)$ for every path $\alpha=e_1\ldots e_n\in E^n$. Throughout the paper, we will consider only countable directed graphs.

Given a directed graph $E$, a {\it Cuntz-Krieger $E$-family} is a collection $\{p_v,s_e:v\in E^0,e\in E^1\}$ of pairwise orthogonal projections $p_v$ and partial isometries $s_e$ with the following relations
\begin{enumerate}
  \item $s_e^*s_e=p_{d(e)}$ for every $e\in E^1$,
  \item $s_es_e^*\leq p_{r(e)}$ for every $e\in E^1$, and
  \item $p_v=\sum_{d(e)=v}s_e s_e^*$ for all vertices $v$ with $0<|d^{-1}(v)|<\infty$.
\end{enumerate}
The {\it graph $C^*$-algebra} $C^*(E)$ is the universal $C^*$-algebra generated by a Cuntz-Krieger $E$-family $\{p_v,s_e\}$ \cite{rae05}. By the above relations, for $e_1,\ldots,e_n\in E^1$, $s_{e_1}\ldots s_{e_n}$ is nonzero if and only if $\alpha:=e_1\ldots e_n$ is a path in $E$; in this case, we write $s_\alpha:=s_{e_1}\ldots s_{e_n}$.

\subsection{Self-similar graphs and their $C^*$-algebras}

Let $G$ be a countable discrete group. An {\it action} $G\curvearrowright E$ is a map $G\times (E^0\cup E^1)\rightarrow E^0\cup E^1$, denoted by $(g,a)\mapsto ga$, such that the action of each $g\in G$ on $E$ gives a graph automorphism.

A {\it self-similar graph} is a triple $(G,E,\varphi)$ such that
\begin{enumerate}
  \item $E$ is a directed graph,
  \item $G$ acts on $E$ by automorphisms, and
  \item $\varphi:G\times E^1\rightarrow G$ is a 1-cocycle for $G\curvearrowright E$ satisfying $\varphi(g,e)v=gv$ for every $g\in G$, $e\in E^1$, and $v\in E^0$.
\end{enumerate}

\begin{rem}
According to \cite[Proposition 2.4]{exe17}, we may extend inductively the action $G\curvearrowright E$ and the cocycle $\varphi$ on the finite path space $E^*$ satisfying the desired relations \cite[Equation 2.6]{exe17}. Indeed, if $\alpha=\alpha_1\alpha_2\in E^*$, then we define
$$g \alpha=(g \alpha_1)(\varphi(g,\alpha_1)\alpha_2) \hspace{5mm}  \mathrm{and} \hspace{5mm} \varphi(g,\alpha)=\varphi(\varphi(g,\alpha_1),\alpha_2).$$
\end{rem}

\begin{defn}[\cite{exe17,exe18}]
Let $(G,E,\varphi)$ be a (countable) self-similar graph. Then $\Og$ is the universal $C^*$-algebra generated by
$$\{p_v,s_e:v\in E^0, e\in E^1\} \cup \{u_g p_v:g\in G,v\in E^0\}$$
satisfying the following properties:
\begin{enumerate}
  \item $\{p_v,s_e:v\in E^0,e\in E^1\}$ is a Cuntz-Krieger $E$-family.
  \item $u:G\rightarrow \mathcal{M}(\Og)$, $g\mapsto u_g$, is a unitary $*$-representation of $G$ on the multiplier algebra $\mathcal{M}(\Og)$.
  \item $u_gp_v=p_{gv}u_g$ for every $g\in G$ and $v\in E^0$.
  \item $u_gs_e=s_{ge}u_{\varphi(g,e)}$ for every $g\in G$ and $e\in E^1$.
\end{enumerate}
We usually use the notation $\Og$ instead of $\mathcal{O}_{(G,E,\varphi)}$ for convenience. Also, we will write each $u_g p_v$ by $u_{gv}$. Then one may easily verify relations (b)-(e) of \cite[Definition 2.2]{exe18}.
\end{defn}

{\bf Standing assumption.} All self-similar graphs $(G,E,\varphi)$ considered in this paper will be countable, row-finite and source-free.

\subsection{The groupoid associated to $(G,E,\varphi)$}

In \cite[Section 4]{exe17}, Exel and Pardo associated an inverse semigroup $\mathcal{S}_{G,E}$ to a self-similar graph $(G,E,\varphi)$ with finite graph $E$. They then showed that $\Og\cong C^*_{tight}(\mathcal{S}_{G,E})\cong C^*(\g)$ where $\g$ is the groupoid of germs for the action of $\mathcal{S}_{G,E}$ on $E^\infty$ \cite[Corollary 6.4 and Proposition 8.4]{exe17}. Note that the constructions of $\mathcal{S}_{G,E}$ and $\g$ in \cite{exe17} may be extended for countable row-finite, source-free self-similar graphs $(G,E,\varphi)$ with small modifications. We give a brief review of it here for convenience. So, fix a row-finite self-similar graph $(G,E,\varphi)$ without sources. Define the $*$-inverse semigroup $\mathcal{S}_{G,E}$ as
$$\mathcal{S}_{G,E}=\{(\alpha,g,\beta):\alpha,\beta\in E^*,g\in G, d(\alpha)=gd(\beta)\}\cup \{0\}$$
with the operations
$$(\alpha,g,\beta)(\gamma,h,\delta):=\left\{
                                      \begin{array}{ll}
                                        (\alpha,g\varphi(h,\varepsilon),\delta h\varepsilon) & \mathrm{if} ~~\beta=\gamma \varepsilon \\
                                        (\alpha g\varepsilon,\varphi(g,\varepsilon)h,\delta) & \mathrm{if} ~~ \gamma=\beta\varepsilon\\
                                        0    & \mathrm{otherwise}
                                      \end{array}
                                    \right.
$$
and $(\alpha,g,\beta)^*:=(\beta,g^{-1},\alpha)$.

Let $E^\infty$ be the space one-sided infinite paths of the form
$$x=e_1e_2\ldots  \hspace{4mm} \mathrm{such ~ that}  \hspace{5mm} d(e_i)=r(e_{i+1}) ~~~ \mathrm{for} ~~~ i\geq 1.$$
By \cite[Proposition 8.1]{exe17}, there is a unique action $G\curvearrowright E^\infty$ as follows: for each $g\in G$ and $x=e_1e_2\ldots \in E^\infty$, there is a unique infinite path $gx=f_1f_2\ldots$ such that
$$f_1f_2\ldots f_n=g(e_1e_2\ldots e_n) \hspace{5mm} (\mathrm{for ~ all}~~ n\geq 1).$$
Moreover, we may consider the action of each $(\alpha,g,\beta)\in \mathcal{S}_{G,E}$ on $x=\beta \hat{x}\in E^\infty$ by $(\alpha,g, \beta)\cdot x=\alpha(g\hat{x})$. Then $\g$ is the groupoid of germs of the action of $\mathcal{S}_{G,E}$ on $E^\infty$, that is
$$\g=\big\{[\alpha,g,\beta;x]: x=\beta\hat{x}\big\}.$$
Recall that two germs $[s;x],[t;y]$ in $\g$ are equal if and only if $x=y$ and there exists an idempotent $0\neq e\in \mathcal{S}_{G,E}$ such that $e\cdot x=x$ and $se=te$. The unit space of $\g$ is
$$\g^{(0)}=\{[\alpha,1_G,\alpha;x]:x=\alpha \hat{x}\},$$
which is identified with $E^\infty$ by $[\alpha,1_G,\alpha;x]\mapsto x$. Then, the range and source maps are defined by
$$r([\alpha,g,\beta;\beta\hat{x}])=\alpha(g\hat{x}) \hspace{5mm} \mathrm{and} \hspace{5mm} s([\alpha,g,\beta;\beta\hat{x}])=\beta\hat{x}.$$
Following \cite[Section 10]{exe17}, we endow $\g$ with the topology generated by compact open bisections of the form
$$\Theta(\alpha,g,\beta;Z(\gamma)):=\{[\alpha,g,\beta;y]\in \g: y\in Z(\gamma)\}$$
where $\gamma\in E^*$ and $Z(\gamma):=\{\gamma x: x\in s(\gamma)E^\infty\}$. Hence, $\g$ is an ample groupoid.

\begin{defn}
We say that $(G,E,\varphi)$ is {\it pseudo free} if for every $g\in G$ and $e\in E^1$,
$$ge=e \hspace{2mm} \mathrm{and} \hspace{2mm} \varphi(g,e)=1_G ~~~ \Longrightarrow ~~g=1_G.$$
\end{defn}

In the end of this section, we recall briefly the following results from \cite{exe17} for convenience. Although they are proved there for finite self-similar graphs with no sources, but we can obtain them for countable cases by a same way (see also \cite{exe18}).

\begin{prop}\label{prop2.7}
Let $(G,E,\varphi)$ be a pseudo free self-similar graphs without sources and let $\g$ be the associated groupoid as above. Then
\begin{itemize}
 \item[(1)] $\g\cong \gr_{tight}(\mathcal{S}_{G,E})$ \cite[Theorem 8.19]{exe17}, $\g$ is Hausdorff \cite[Proposition 12.1]{exe17}, and $\Og\cong C^*(\g)$ \cite[Theorem 9.6]{exe17}.
 \item[(2)] If moreover $G$ is an amenable group, then $\g$ is an amenable groupoid in the sense of \cite{ana00}. In particular, we have $\Og\cong C^*(\g)\cong C_r^*(\g)$ by \cite[Proposition 6.1.8]{ana00}.
\end{itemize}
\end{prop}

\begin{prop}[{\cite[Corollary 14.13]{exe17} and \cite[Theorem 4.4]{exe18}}]\label{prop2.8}
Let $(G,E,\varphi)$ be a pseudo free self-similar graph with no sources. Then $\g$ is effective\footnote{Note that the `effective' property of groupoids is called {\it essentially principal} in \cite{exe17, exe18}.} if and only if the following properties hold:
\begin{itemize}
  \item[(1)] Every $G$-circuit in $E$ has an entry, and
  \item[(2)] for every $v\in E^0$ and $1_G\neq g\in G$, the action of $g$ on $Z(v)$ is nontrivial (i.e., there is $x\in Z(v)$ such that $g.x\neq x$).
\end{itemize}
\end{prop}

\section{Purely infinite self-similar graph $C^\ast$-algebras}

In \cite[Corollary 16.3]{exe17} and \cite[Corollary 4.7]{exe18}, it is shown that when $\Og$ is simple and $(G,E,\varphi)$ contains a $G$-circuit, then $\Og$ is purely infinite. In this section, we study purely infinite $C^*$-algebras $\Og$ of countable self-similar graphs in the sense of \cite{ror02} without the simplicity assumption. Our main result here is a generalization of \cite[Theorem 16.2]{exe17} to countable self-similar graphs. Note that there is another well-known notion of pure infiniteness from \cite{kirch00} which is equivalent to that of \cite{ror02} for the simple cases. Moreover, our results in this section may be generalized for the Kirchberg-R{\o}rdam's notion using \cite[Corollary 3.15]{kirch00} and the ideal structure \cite[Corollary 6.15]{lal19}.

\begin{thm}\label{thm3.1}
Let $(G,E,\varphi)$ be a pseudo free self-similar graph over an amenable group $G$. Suppose also that $(G,E,\varphi)$ satisfies conditions (1) and (2) of Proposition \ref{prop2.8} (i.e., the groupoid $\g$ is effective). Then $\Og$ is purely infinite if and only if every vertex projection $s_v$ is infinite in $\Og$.
\end{thm}

\begin{proof}
We must prove the ``if" implication only. So suppose that for every $v\in E^0$, $s_v$ is infinite in $\Og$. Let $\mathcal{G}=\g$ be the groupoid associated to $(G,E,\varphi)$. By Proposition \ref{prop2.7}(2), $\mathcal{G}$ is amenable, so $C^*_r(\mathcal{G})=C^*(\mathcal{G})=\Og$. We know that the cylinders $\{Z(\alpha):\alpha\in E^*\}$ is a basis of compact open sets for the topology induced on $E^\infty=\mathcal{G}^{(0)}$. Moreover, Proposition \ref{prop2.8} says that $\mathcal{G}$ is effective. Hence, Proposition \ref{prop2.3} implies that $\Og=C^*_r(\mathcal{G})$ is purely infinite if and only if $\{1_{Z(\alpha)}=s_\alpha s_\alpha^*:\alpha\in E^*\}$ are all infinite projections in $\Og$. Now since $s_\alpha s_\alpha^*\sim s_\alpha^*s_\alpha=s_{d(\alpha)}$ and the infiniteness passes through Murray-von Neumann equivalence, we conclude the result.
\end{proof}

\begin{defn}
Let $v,w\in E^0$. We say that $v$ receives a {\it $G$-path from $w$} or $w$ {\it connects to $v$ by a $G$-path}, say $v\gtrsim w$, if there exist $\alpha\in E^*$ and $g\in G$ such that $r(\alpha)=v$ and $d(\alpha)=gw$. By \cite[Proposition 13.2]{exe17}, this is equivalent to
$$\exists \alpha\in E^*, ~~ \exists g\in G \hspace{3mm} \mathrm{such~ that} \hspace{3mm} r(\alpha)=gv ~~\mathrm{and} ~~ d(\alpha)=w.$$
\end{defn}

\begin{lem}\label{lem3.3}
Let $(G,E,\varphi)$ be a self-similar graph. For $v,w\in E^0$ and $\alpha,\beta\in E^*$, we have
\begin{itemize}
  \item[(1)] If $v=gw$ for some $g\in G$, then $s_v\sim s_w$ in the Murray-von Neumann sense.
  \item[(2)] If $v$ receives a $G$-path from $w$, then $s_v\succsim s_w$.
  \item[(3)] If $\beta=g\alpha$ for some $g\in G$, then $s_\beta s_\beta^* \sim s_\alpha s_\alpha^*$.
\end{itemize}
\end{lem}

\begin{proof}
(1). If $v=gw$, then we have $s_v=(u_gs_v)^*(u_g s_v)$ and
$$(u_gs_v)(u_g s_v)^*=(s_{gv}u_g)(s_{gv}u_g)^*=s_wu_gu_g^* s_w=s_w,$$
concluding $s_v\sim s_w$.

For (2), suppose that there exist $\alpha\in E^*$ and $g\in G$ such that $r(\alpha)=v$ and $d(\alpha)=gw$. Then, by the Cuntz-Krieger relations,
$$s_v\geq s_\alpha s_\alpha^* \sim s_\alpha^* s_\alpha=s_{d(\alpha)}=s_{gw}\sim s_w,$$
and consequently $s_v \succsim s_w$.

For (3), if $\beta=g\alpha$, then by part (1) we have
$$s_\beta s_\beta^*\sim s_\beta^* s_\beta=s_{d(\beta)}=s_{g.d(\alpha)}\sim s_{d(\alpha)}=s_\alpha^* s_\alpha \sim s_\alpha s_\alpha^*,$$
giving $s_\beta s_\beta^*\sim s_\alpha s_\alpha^*$.
\end{proof}

\begin{prop}\label{prop3.4}
Let $(G,E,\varphi)$ be a pseudo free self-similar graph over an amenable group $G$. Suppose that conditions (1) and (2) of Proposition \ref{prop2.8} hold. Then
\begin{itemize}
  \item[(1)] If every $v\in E^0$ receives a $G$-path from a $G$-circuit, then $\Og$ is purely infinite.
  \item[(2)] If the graph $C^*$-algebra $C^*(E)$ is purely infinite, then so is $\Og$.
\end{itemize}
\end{prop}

\begin{proof}
(1). In view of Theorem \ref{thm3.1}, it suffices to prove that each $s_v$ is infinite in $\Og$. So, fix some $v\in E^0$. By hypothesis, there is a $G$-circuit $\alpha$ connecting to $v$ by a $G$-path.

We first show that $s_{r(\alpha)}$ is infinite. For, let $\gamma$ be an entry for $\alpha$ by assumption. Since each of $\alpha$ nor $\gamma$ is not a subpath of the other, one may compute that $s_\alpha s_\alpha^*$ and $s_\gamma s_\gamma^*$ are orthogonal. Hence, the Cuntz-Krieger relations imply that
$$s_{r(\alpha)}\geq s_\alpha s_\alpha^*+s_\gamma s_\gamma^*> s_\alpha s_\alpha^*\sim s_\alpha^* s_\alpha=s_{d(\alpha)}.$$
If $d(\alpha)=gr(\alpha)$, then $s_{d(\alpha)}\sim s_{r(\alpha)}$ by Lemma \ref{lem3.3}(1), and whence $s_{r(\alpha)}$ is infinite in $\Og$ as claimed.

Now, because there is a $G$-path from $r(\alpha)$ to $v$, we have $s_v\succsim s_{r(\alpha)}$ by Lemma \ref{lem3.3}(2), and therefore $s_v$ is infinite as well. As $v\in E^0$ was arbitrary, Theorem \ref{thm3.1} follows the result.

(2). If $C^*(E)$ is purely infinite, then each $s_v$ is infinite in $C^*(E)$, and so is in $\Og$ as well. Now apply Theorem \ref{thm3.1}.
\end{proof}

\begin{rem}
If $v\in E^0$ receives a $G$-path from a $G$-circuit  with an entry but not a path from a circuit, then $s_v$ is infinite in $\Og$ while not in $C^*(E)$. Therefore, the converse of Proposition \ref{prop3.4}(2) does not necessarily hold.
\end{rem}

In the simple case we conclude the following.

\begin{cor}\label{cor3.6}
Let $(G,E,\varphi)$ be a pseudo free self-similar graph over an amenable group $G$. Suppose that $\Og$ is simple. If $E$ contains a $G$-circuit, then $\Og$ is purely infinite.
\end{cor}

\begin{proof}
Note that the simplicity of $\Og$ gives conditions (1) and (2) in Proposition \ref{prop2.8} \cite[Theorem 4.5]{exe18}. So, by Theorem \ref{thm3.1}, it suffices to show that $s_v$ is infinite for each $v\in E^0$.

Let $(g,\alpha)$ be a $G$-circuit in $E$. By \cite[Theorem 16.1]{exe17}, $(g,\alpha)$ has an entry, hence $s_{r(\alpha)}$ is infinite as seen in the proof of Proposition \ref{prop3.4}(1).

Fix an arbitrary $v\in E^0$. We may form the infinite path $\alpha^\infty=\alpha(g\alpha)(g^2 \alpha)\cdots$, which is well-defined because
$$d(g^n\alpha)=g^n d(\alpha)=g^n g r(\alpha)=r(g^{n+1}\alpha).$$
Since $E$ is also weakly $G$-transitive by \cite[Theorem 4.5]{exe18}, there is a $G$-path from $r(g^n\alpha)$ to $v$ for sufficiently large $n$. Note that as $r(g^n\alpha)=g^n r(\alpha)$, $s_{r(g^n \alpha)}=s_{g^n r(\alpha)}$ is infinite by Lemma \ref{lem3.3}(1). Also, Lemma \ref{lem3.3}(2) implies that $s_v\succsim s_{r(g^n \alpha)}\sim s_{r(\alpha)}$, and consequently $s_v$ is infinite too. As $v\in E^0$ was arbitrary, Theorem \ref{thm3.1} concludes that $\Og$ is purely infinite.
\end{proof}

\begin{rem}
The converse of above corollary will be proved in Theorem \ref{thm4.9} (1) $\Longleftrightarrow$ (6).
\end{rem}

The following result gives necessary and sufficient criteria for the purely infinite simple $C^*$-algebras by the monoiod of equivalent projections. It is new even for the ordinary graph $C^*$-algebras. Before that we recall the definition of $K_0$-group of a unital $C^*$-algebra and establish some notations. Let $A$ be a unital $C^*$-algebra and write by $\mathcal{P}(A)$ the collection of all projections in $M_\infty(A)=\bigcup_{n\geq 1}M_n(A)$. We say that two projections $p\in M_m(A)$ and $q\in M_n(A)$ are equivalent, denoted by $p\sim q$, if
$$\exists~ v\in M_{m,n}(A) \hspace{2mm} \mathrm{such ~ that} \hspace{2mm} p=v^*v \hspace{2mm} \mathrm{and} \hspace{2mm} q=v^*v.$$
Note that, if $m\leq n$, then $p\sim q$ if and only if $p\oplus 0_{n-m}$ is Murray-von Neumann equivalent to $q$ in $M_n(A)$, where $x\oplus y:=\mathrm{diag}(x,y)$. Define $\mathcal{D}(A):=\mathcal{P}(A)/\sim=\{[p]:p\in\mathcal{P}(A)\}$, which is an abelian monoid with the operation $[p]+[q]:=[p\oplus q]$. Then $K_0(A)$ is {\it the Grothendieck group} of $\mathcal{D}(A)$ endowed with a universal Grothendieck map $\phi:\mathcal{D}(A)\rightarrow K_0(A)$. The image of $\mathcal{D}(A)$ under $\phi$ is denoted by $K_0(A)^+$. It is known that when $\mathcal{D}(A)\setminus \{0\}$ is a group, then $K_0(A)=\mathcal{D}(A)\setminus\{0\}$.

\begin{thm}\label{thm3.7}
\begin{itemize}
  \item[(1)] Let $E$ be an arbitrary directed graph (non necessarily row-finite, source-free, or even countable) with $|E^0|<\infty$. Then $C^*(E)$ is purely infinite and simple if and only if $\mathcal{D}(C^*(E))\setminus \{0\}$ is a group (or equivalently, $\mathcal{D}(C^*(E))\setminus \{0\}=K_0(C^*(E))$).
  \item[(2)] Let $(G,E,\varphi)$ be a pseudo free self-similar graph over an amenable group $G$. Suppose also that $|E^0|<\infty$ and conditions (1) and (2) of Proposition \ref{prop2.8} hold. Then $\Og$ is purely infinite simple if and only if $\mathcal{D}(\Og)$ is a group.
\end{itemize}
\end{thm}

\begin{proof}
Note that the ``only if" implications hold for every unital purely infinite simple $C^*$-algebra. Indeed, if $A$ is a purely infinite simple $C^*$-algebra, then nonzero projections of $A$ are all infinite. Thus, combining Proposition 1.5 and Theorem 1.4 of \cite{cun81} implies that $\mathcal{D}(A)\setminus \{0\}$ is a group ($=K_0(A)$).

So it is enough to prove the ``if" parts. We first show that every projection $p$ in $A$ is infinite for any unital $C^*$-algebra $A$ with $\mathcal{D}(A)\setminus \{0\}$ a group. Indeed, if $[f]$ is the identity of $\mathcal{D}(A)\setminus\{0\}$, then
$$[p]=[p]+[f]=[p\oplus f],$$
thus we have
$$p\sim p\oplus 0<p\oplus f\sim p,$$
where $0$ is a zero matrix in $M_\infty(A)$. Therefore, $p$ is an infinite projection in $A$, as claimed.

In the case of statement (1), this follows that $E$ satisfies Condition (L). In fact if there exists a circuit in $E$ with no entries, then $C^*(E)$ contains an ideal Morita equivalent to $C(\mathbb{T})$, hence it has a finite projection. Recall that by Condition (L) every ideal of $C^*(E)$ has a (vertex) projection. Now take a nonzero ideal $I$ of $C^*(E)$ and some projection $0\neq p\in I$. As $|E^0|<\infty$, write $1:=\sum_{v\in E^0}s_v$ the unit of $C^*(E)$. Then $[p]+[1-p]=[1]$ and we have
$$[p]=[1]+[q]=[1\oplus q],$$
where $[q]$ is the inverse of $[1-p]$ in $\mathcal{D}(C^*(E))\setminus \{0\}$. Therefore, $p\sim 1\oplus q$ which says that there is $x=[x_1 x_2 x_3 \ldots]\in M_{1,\infty}(\Og)$ such that $x^* p x=1\oplus q$. In particular, $1=x_1^* p x_1\in I$, concluding $I=C^*(E)$. Therefore $C^*(E)$ is simple.

For the pure infiniteness, let $B$ be a nonzero hereditary $C^*$-subalgebra of $C^*(E)$. Again, Condition (L) gives a nonzero projection $p$ in $B$. If $[f]$ is the identity of $\mathcal{D}(C^*(E))\setminus\{0\}$, then
$$[p]=[p]+[f]=[p\oplus f],$$
and we have
$$p\sim p\oplus 0<p\oplus f\sim p$$
where $0$ is a zero matrix in $M_\infty(\Og)$, and consequently $p$ is infinite. Therefore, $C^*(E)$ is purely infinite.

For statement (2), note that $\g$ is effective by Proposition \ref{prop2.8}, and $\Og\cong C^*_r(\g)$ by Proposition \ref{prop2.7}. This implies that every ideal of $\Og$ contains a projection (see \cite[Theorem4.4]{exe10} for example). Now we may follow the proof of statement (1) to obtain the result.
\end{proof}


\section{Stable finiteness and a dichotomy}

In this section, we associate a special graph $\e$ to any self-similar graph $(G,E,\varphi)$. We show that if the graph $C^*$-algebra $C^*(\e)$ is either simple, purely infinite, or stable finite then so is $\Og$ respectively. Then we will conclude a dichotomy for simple self-similar graph $C^*$-algebras.

\begin{defn}
Let $\mathbb{K}$ denote the $C^*$-algebra of compact operators on a separable, infinite dimensional Hilbert space. A (simple) $C^*$-algebra $A$ is called {\it stably finite} if $A\otimes \mathbb{K}$ contains no infinite projections.
\end{defn}

Fix a self-similar graph $(G,E,\varphi)$. In the following we define a graph $\widetilde{E}$ associated to $(G,E,\varphi)$. Define $\approx$ on $E^*=\bigsqcup_{n=0}^\infty E^n$ by
$$\alpha\approx \beta \hspace{5mm} \Longleftrightarrow \hspace{5mm} \exists g\in G ~~ \mathrm{such ~ that} ~~ \beta=g\alpha,$$
which is an equivalent relation on each $E^n$ (and so on $E^*$). The vertex set of $\e$ is $\e^0:=E^0/\approx$ the collection of vertex classes. In each class $[v] \in \e^0$ pick exactly one vertex up and collect them in the set $\Omega$. Hence, $\e^0=\{[v]:v\in \Omega\}$, and we have $[v]\ne [w]$ for $v\ne w\in \Omega$. For every $v\in \Omega$ and $e\in r^{-1}(v)$ draw an edge $\tilde{e}$ from $[d(e)]$ to $[v]$. Hence we obtain the graph $\e$ so that
\begin{align*}
\e^0&:=\{[v]:v\in \Omega\}, ~\mathrm{and}\\
\e^1&:=\bigcup_{v\in\Omega}\widetilde{r^{-1}(v)}=\bigcup_{v\in\Omega}\{\widetilde{e}:r(e)=v\},
\end{align*}
with the range $\widetilde{r}(\widetilde{e})=[r(e)]$ and domain $\widetilde{d}(\widetilde{e})=[d(e)]$ for every $\widetilde{e}\in \e^1$.

\begin{ex}\label{ex4.2}
For $n\geq1$, let $\mathbb{Z}_{\mathrm{mod}n}$ be the additive group $\{1,2,\ldots, n\}$. Let $(\mathbb{Z}_{\mathrm{mod}n},E,\varphi)$ be a triple with the cyclic graph $E$

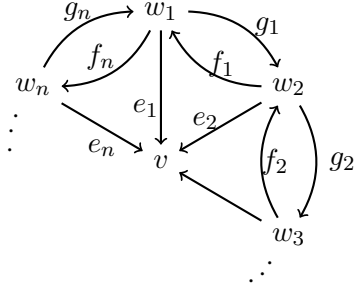
\begin{center}
\begin{tikzpicture}[thick]
\node (v) at (0,0) {$v$};
\node (w1) at (0,2) {$w_1$};
\node (w2) at (1.7,1) {$w_2$};
\node (w3) at (1.7,-1) {$w_3$};
\node (wn) at (-1.7,1) {$w_n$};

\node () at (.8,1.3) {$f_1$};
\node () at (-.8,1.4) {$f_n$};
\node () at (-1.1,2) {$g_n$};

\node () at (1.45,-1.35) {$\cdot$};
\node () at (1.35,-1.5) {$\cdot$};
\node () at (1.2,-1.6) {$\cdot$};

\node () at (-1.93,.6) {$\cdot$};
\node () at (-1.99,.4) {$\cdot$};
\node () at (-2.01,.15) {$\cdot$};

\path[<-] (v) edge node[above=-15pt] {$e_1\hspace{4mm}$} (w1);
\path[<-] (v) edge node[above=-5pt] {$\hspace{-4mm}e_2$} (w2);
\path[<-] (v) edge node[above=-5pt] {$$} (w3);
\path[<-] (v) edge node[above=-15pt] {$e_n$} (wn);
\path[<-] (w1) edge[bend right] node[right=-5pt] {$$} (w2);
\path[<-] (w2) edge[bend right] node[right=1pt] {$g_1$} (w1);
\path[<-] (w2) edge[bend right] node[right=-4pt] {$f_2$} (w3);
\path[<-] (w3) edge[bend right] node[right=1pt] {$g_2$} (w2);
\path[<-] (wn) edge[bend right] node[right=-5pt] {$$} (w1);
\path[<-] (w1) edge[bend right] node[left=0pt] {$$} (wn);

\end{tikzpicture}

\end{center}
and the action $\mathbb{Z}_{\mathrm{mod}n}\curvearrowright E$ defined by
$$kv:=v \hspace{5mm} \mathrm{and} \hspace{5mm} k\alpha_i:=\alpha_{k+i} \hspace{10mm} (1\leq k,i \leq n), $$
for every $\alpha_i\in \{w_i, e_i, f_i, g_i\}$. Since $w_i\approx w_j$, for any $1\leq i,j\leq n$, we may select $w_1$ of the class $[w_1]=\{w_1,\ldots, w_n\}$. As $r^{-1}(v)=\{e_1,\ldots,e_n\}$ and $r^{-1}(w_1)=\{f_1,g_n\}$, then the graph $\e$ would be

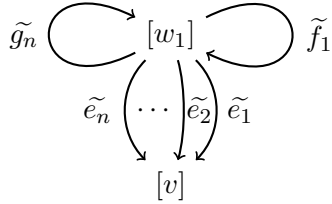
\begin{center}
\begin{tikzpicture}[thick]
\node (v) at (0,0) {$[v]$};
\node (w1) at (0,2) {$[w_1]$};

\node () at (-.2,1) {$\cdots$};

\draw[->] (w1) ..controls (.7,1.3) and (.7,.7) .. node[right=0pt] {$\widetilde{e_1}$} (v);
\draw[->] (w1) ..controls (.2,1.3) and (.2,.7) .. node[right=-3pt] {$\widetilde{e_2}$} (v);
\draw[->] (w1) ..controls (-.7,1.3) and (-.7,.7) .. node[left=1pt] {$\widetilde{e_n}$} (v);

\draw[<-] (w1) ..controls (2,1) and (2,3) .. node[right=1pt] {$\widetilde{f_1}$} (w1);
\draw[->] (w1) ..controls  (-2,1) and (-2,3) .. node[left=0pt] {$\widetilde{g_n}$} (w1);

\end{tikzpicture}
\end{center}

\end{ex}

\begin{lem}\label{lem4.3}
Let $(G,E,\varphi)$ be a self-similar graph, and consider an associated graph $\e$ as above. Then
\begin{itemize}
  \item[(1)] If $E$ is row-finite, then so is $\e$.
  \item[(2)] For each finite path $\widetilde{\alpha}=\widetilde{\alpha}_1\ldots \widetilde{\alpha}_n\in \e^n$, there is a path $\gamma=\gamma_1\ldots \gamma_n$ in $E^n$ such that $\gamma\approx \alpha_i$ for $1\leq i\leq n$. Conversely, if $\gamma=\gamma_1\ldots \gamma_n\in E^n$, then there exists $\widetilde{\alpha}=\widetilde{\alpha}_1\ldots \widetilde{\alpha}_n\in \e^n$ such that $\gamma\approx \alpha_i$ for $1\leq i\leq n$.
  \item[(3)] If $\widetilde{\alpha}\in \e^n$ and $\gamma\in E^n$ are two paths as in statement (2), then $\widetilde{\alpha}$ is a circuit in $\e$ if and only if $\gamma$ is a $G$-circuit in $E$. Moreover, $\widetilde{\alpha}$ has an entry if and only if $\gamma$ does.
\end{itemize}
\end{lem}

\begin{proof}
Statement (1) is clear by the definition of $\e$. For (2), let first $\widetilde{\alpha}=\widetilde{\alpha}_1\ldots \widetilde{\alpha}_n\in \e^n$ be a path in $\e$. Then, for each $1\leq i<n$, we have
$$[d(\alpha_i)]=\widetilde{d}(\widetilde{\alpha}_i)=\widetilde{r}(\widetilde{\alpha}_{i+1})=[r(\alpha_{i+1})],$$
and so there exists $g_i\in G$ such that $d(\alpha_i)=g_i r(\alpha_{i+1})$. Now set $\gamma_1:=\alpha_1$ and $\gamma_i:=g_1\ldots g_{i-1}\alpha_i$ for every $2\leq i\leq n$. Then
$$d(\gamma_i)=d(g_1\ldots g_{i-1}\alpha_i)=g_1\ldots g_{i-1}d(\alpha_i)=g_1\ldots g_{i-1} g_i r(\alpha_{i+1})=r(\gamma_{i+1}),$$
and hence $\gamma=\gamma_1\ldots \gamma_n$ is a desired path in $E$.

Conversely, let $\gamma=\gamma_1\ldots \gamma_n$ be a finite path in $E^n$. For each $1\leq i\leq n$, there is $v_i\in \Omega$ such that $v_i=g_i r(\gamma_i)$ for some $g_i\in G$. Hence, we have $\widetilde{\alpha}=(\widetilde{g_1 \gamma_1})\ldots (\widetilde{g_n \gamma_n})\in \e$ with $\alpha \approx \gamma$.

For statement (3), given $\widetilde{\alpha}$ and $\gamma$ as in part (2), we have
\begin{align*}
\widetilde{\alpha} ~ \mathrm{is ~ a ~ circuit ~ in ~ } E &\Longleftrightarrow ~ ~ [d(\alpha_n)]=[r(\alpha_1)]\\
&\Longleftrightarrow ~ ~ d(\alpha_n)\approx r(\alpha_1)\\
&\Longleftrightarrow ~ ~  d(\gamma_n)\approx d(\alpha_n)\approx r(\alpha_1) \approx r(\gamma_1)\\
&\Longleftrightarrow ~ ~ \gamma ~ \mathrm{is ~ a ~} G\mathrm{-circuit}.
\end{align*}
Moreover, since $|r^{-1}(r(\gamma_i))|=|\widetilde{r}^{-1}(\widetilde{r}(\widetilde{\alpha}_i))|$ for each $1\leq i\leq n$, we have
\begin{align*}
\gamma \mathrm{ ~ has ~ an ~ entry} \hspace{2mm} &\Longleftrightarrow \hspace{2mm} |r^{-1}(r(\gamma_i))|>1 ~~ \mathrm{for ~ some~ } 1\leq i\leq n\\
&\Longleftrightarrow \hspace{2mm} |\widetilde{r}^{-1}(\widetilde{r}(\alpha_i))|>1 ~~ \mathrm{for ~ some~ } 1\leq i\leq n\\
&\Longleftrightarrow \hspace{2mm} \widetilde{\alpha} \mathrm{~has ~ an ~ entry ~ in ~ } \widetilde{E} .
\end{align*}
\end{proof}

\begin{defn}
Let $(G,E,\varphi)$ be a self-similar graph. Following \cite[Definition 3.4]{exe17}, we say that $E$ is {\it weakly $G$-transitive} if for every $v\in E^0$ and $x\in E^\infty$, there exists a path $\alpha$ such that $d(\alpha)=x(n,n)$ for some $n\geq 0$ and $r(\alpha)=g v$ for some $g\in G$. If we have an ordinary graph $E$ (with the trivial group action), we say simply that $E$ is {\it weakly transitive}. Note that the weakly transitive is called {\it cofinal} in \cite{rae05}.
\end{defn}

\begin{lem}\label{lem4.5}
Let $(G,E,\varphi)$ be a self-similar graph, and associate a graph $\e$ as above. Then
\begin{itemize}
  \item[(1)] Every $G$-circuit in $E$ has an entry if and only if every circuit in $\e$ does.
  \item[(2)] $E$ is weakly $G$-transitive if and only if $\e$ is weakly transitive.
\end{itemize}
\end{lem}

\begin{proof}
Statement (1) follows from items (2) and (3) of Lemma \ref{lem4.3}. For (2), let $\e$ be transitive. Take an arbitrary infinite path $x\in E^\infty$ and some $v\in E^0$. By item (2) in Lemma \ref{lem4.3}, there is $\widetilde{y}\in \e^\infty$ such that $y(0,n)\approx x(0,n)$ for every $n\geq 0$. By transitivity, there exists $\widetilde{\gamma}\in \e^*$ such that $\widetilde{r}(\widetilde{\gamma})=[v]$ and $\widetilde{d}(\widetilde{\gamma})=[y(n,n)]$ for some $n$. Hence, $v\approx r(\gamma)$ and $d(\gamma)\approx y(n,n)\approx x(n,n)$. This follows that $E$ is $G$-transitive. The converse is analogous.
\end{proof}

\begin{prop}\label{prop4.6}
Let $(G,E,\varphi)$ be a self-similar graph over an amenable group $G$, and let $\e$ be an associated graph.
\begin{itemize}
  \item[(1)] In case the groupoid $\g$ is Hausdorff (see \cite[Theorem 4.2]{exe18}), then $\Og$ is simple if and only if
    \begin{itemize}
      \item[(a)] the graph $C^*$-algebra $C^*(\e)$ is simple, and
      \item[(b)] for $v\in E^0$ and $g\in G$, if the action of $g$ on the cylinder $Z(v)$ is trivial (i.e., $g x=x$ for every $x\in Z(v)$), then $g$ is slack at $v$.
    \end{itemize}
  \item[(2)] Suppose that $(G,E,\varphi)$ is pseudo free and for any $v\in E^0$ and $1_G\ne g\in G$, the action of $g$ on $Z(v)$ is nontrivial. If $C^*(\e)$ is purely infinite, then so is $\Og$.
\end{itemize}
\end{prop}

\begin{proof}
Statement (1) follows from Lemma \ref{lem4.5} and \cite[Theorem 4.5]{exe18}. For (2), if the graph $C^*$-algebra $C^*(\e)$ is purely infinite, then every circuit in $\e$ has an entry and every vertex $[v]\in \e^0$ can be reached from a circuit. By Lemma \ref{lem4.5}, every $G$-circuit has an entry and every $v\in E^0$ receives a $G$-path from a $G$-circuit. Now, Proposition \ref{prop3.4}(1) concludes that $\Og$ is purely infinite.
\end{proof}

\begin{ex}
The graph $\e$ in Example \ref{ex4.2} is weakly transitive and every circuit in $\e$ has an entry. Then $C^*(\e)$ is simple and purely infinite, and so is the $C^*$-algebra $\Og$ by Proposition \ref{prop4.6}.
\end{ex}

\begin{defn}[\cite{hje01}]
Let $(G,E,\varphi)$ be a self-similar graph. A {\it graph trace} on $E$ is map $T:E^0\rightarrow \mathbb{R}^+$ such that
\begin{enumerate}
  \item $T(r(e))\geq T(d(e))$ for every $e\in E^1$, and
  \item $T(v)=\sum_{r(e)=v}T(d(e))$ for every $v\in E^0$.
\end{enumerate}
A {\it graph $G$-trace} in $E$ is a graph trace $T:E^0\rightarrow \mathbb{R}^+$ such that $T(v)=T(w)$ for every $v\approx w$ in $E^0$.
\end{defn}

\begin{thm}\label{thm4.9}
Let $(G,E,\varphi)$ be a pseudo free self-similar graph over an amenable group $G$. Suppose that $\Og$ is simple. Then the following are equivalent.
\begin{itemize}
  \item[(1)] $\Og$ is stably finite.
  \item[(2)] $\Og$ is quasi diagonal.
  \item[(3)] $(G,E,\varphi)$ has a nonzero graph $G$-trace.
  \item[(4)] $\e$ has a nonzero graph trace.
  \item[(5)] $\e$ contains no circuits.
  \item[(6)] $E$ contains no $G$-circuits.
\end{itemize}
\end{thm}

\begin{proof}
Statements (1) and (2) are equivalent by \cite[Corollary 6.6]{rai18}.

(1) $\Rightarrow$ (6). If $E$ has a $G$-circuit, then $\Og$ is purely infinite by Corollary \ref{cor3.6}. In particular, $\Og$ is not stably finite, a contradiction.

(6) $\Rightarrow$ (5) follows from Lemma \ref{lem4.3}(3).

(5) $\Rightarrow$ (4). Suppose that $\e$ has no circuits. Arrange $\e^0=\{[v_1],[v_2],\ldots\}$. For each $n\geq 1$, let $F_n$ be the full subgraph of $\e$ containing all $\bigcup_{i=1}^n \widetilde{r}^{-1}([v_i])$. Since $F_n$'s have no circuits, \cite[Corollary 2.3]{kum98} implies that $C^*(F_1)\subseteq C^*(F_2) \subseteq \ldots$ is a sequence of finite dimensional $C^*$-subalgebras of $C^*(\e)$ such that $C^*(\e)=\lim C^*(F_n)$ (i.e., $C^*(\e)$ is AF). Thus there exist bounded traces $\tau_n:C^*(F_n)\rightarrow \mathbb{C}$ such that $\tau_n|_{C^*(F_i)}$ equals with $\tau_i$ for $i\leq n$. This induces a semifinite trace $\tau=\lim \tau_n$ on $C^*(\e)$. Therefore, if $C^*(\e)=C^*(t_e,q_{[v]})$, we obtain the nonzero graph trace $T:\e^0\rightarrow \R^+$, by $T([v])=\tau(q_{[v]})$, on $\e$.

(4) $\Rightarrow$ (3). Suppose that $T$ is a nonzero graph trace on $\e$. Note that, since the action of $G$ on $E^1$ gives automorphisms respecting to the range and domain, for any $v\neq w\in E^0$ with $w=g v$, the map $e\mapsto g e$ is a bijection from $r^{-1}(v)$ onto $r^{-1}(w)$. In particular, $|r^{-1}(w)|=|r^{-1}(v)|$. Being this fact in mind, one may easily see that the map $T':E^0\rightarrow \R^+$, defined by $T'(v):=T([v])$, is a nonzero graph $G$-trace on $E$, as desired.

(3) $\Rightarrow$ (1). By \cite[Proposition II.4.8]{ren80}, there exists a faithful conditional expectation $\pi:C^*(\g)\rightarrow C_0(\g^{(0)})$ such that $\pi(f)=f|_{\g^{(0)}}$ for all $f\in C_c(\g^{(0)})$. Note that the isomorphism $\psi:\Og \rightarrow C^*(\g)$ in Proposition \ref{prop2.7}(1) maps the core $\Og^0:=\overline{\mathrm{span}}\{s_\alpha s_\alpha^*:\alpha\in E^*\}$ onto $C_0(\g^{(0)})$. Hence $\varphi:=\psi^{-1}\circ \pi\circ \psi$ is a faithful conditional expectation from $\Og$ onto $\Og^0$ such that
$$\phi(s_\alpha u_g s_\beta^*)=\left\{
                                 \begin{array}{ll}
                                   s_\alpha s_\alpha^* & \beta=\alpha,~ g=1_G \\
                                   0 & \mathrm{otherwise}
                                 \end{array}
                               \right.
$$
for every $\alpha,\beta\in E^*$ and $g\in G$.

Now suppose that $T$ is a nonzero graph $G$-trace on $E$. Define $t:\Og^0\rightarrow \C$ by $t(s_\alpha s_\alpha^*)=T(d(\alpha))$, which is a linear functional on $\Og^0$. So, we may easily verify that $\tau:=t\circ \phi$ is a semifinite trace on $\Og$ such that $0<\tau(s_v)<\infty$ for all $v\in E^0$. Moreover, $\tau$ is faithful because $\Og$ is simple. Thus \cite[Corollary 6.6]{rai18} yields that $\Og$ is stably finite.
\end{proof}

Recall from  \cite[Corollary 10.16]{exe17} that if $G$ is amenable, then $\Og$ is a nuclear $C^*$-algebra. So, combining Corollary \ref{cor3.6} and Theorem \ref{thm4.9} implies the following dichotomy for simple $\Og$.

\begin{cor}\label{cor4.10}
Let $(G,E,\varphi)$ be a pseudo free self-similar graph over an amenable group $G$. Suppose that $\Og$ is simple. Then
\begin{itemize}
  \item[(1)] If $E$ has a $G$-circuit, then $\Og$ is purely infinite. In this case, $\Og$ is a Kirchberg algebra, and we have $K_0(\Og)= D(\Og)\setminus \{0\}$ whenever $|E^0|<\infty$.
  \item[(2)] Otherwise, $\Og$ is stably finite. In this case, $(K_0(\Og),K_0(\Og)^+)$ is an ordered abelian group (see \cite[Proposition 5.1.5(iv)]{ror00}).
\end{itemize}
\end{cor}

\begin{rem}
Note that in case $\Og$ is stably finite, the embedding $\iota:C^*(E)\hookrightarrow \Og$ of \cite[Section 11]{exe17} induces an embedding $K_0(\iota):K_0(C^*(E))\hookrightarrow K_0(\Og)$ defined by $K_0(\iota)([p]_0):=[\iota(p)]_0$, where the map $\iota$ is naturally extended on $M_\infty(C^*(E))$ into $M_\infty(\Og)$. Indeed, if $p\in M_\infty(C^*(E))$ is a projection with $[\iota(p)]_0=0$, then we must have $\iota(p)=0$ because $M_\infty(\Og)$ has no infinite projection, and hence $p=0$.
\end{rem}


\section{Pure infiniteness of self-similar $k$-graph $C^*$-algebras}

In this section, we consider the pure infiniteness of self-similar $k$-graph $C^*$-algebras. Let us first recall the definitions of self-similar $k$-graphs and their $C^*$-algebras from \cite{li18}. Fix $k\in\N\cup\{\infty\}$ and let $\La=(\La^0,\La,r,s)$ be a row-finite $k$-graph with no sources (we refer the reader to \cite{rae05} for basic definitions and concepts about $k$-graphs and associated $C^*$-algebras). Consider $\N^k$ as a category with a single object $0$ and the coordinatewise partial order $\leq$. Let $\Omega_k:=\{(p,q):p,q\in \N^k, p\leq q\}$. An {\it infinite path in $\La$} is a morphism $x:\Omega_k\rightarrow\La$ with the range $r(x):=x(0,0)$. We write by $\La^\infty$ the set of infinite paths in $\La$.

Let $G$ be a (discrete and countable) group. {\it An action $G\curvearrowright \La$} is a map $G\times \La\rightarrow \La$, $(g,\la)\rightarrow g \la$, which gives a graph automorphism preserving the degree map for every $g\in G$.

\begin{defn}[\cite{li18}]\label{defn5.1}
A {\it self-similar $k$-graph} is a triple $(G,\La,\varphi)$, where $\La$ is a $k$-graph, $G$ is a group acting on $\La$, and $\varphi:G\times \La\rightarrow \La$ is a cocycle for $G\curvearrowright \La$ with the property
$$\varphi(g,\la).v=g v \hspace{5mm} (g\in G,v\in \La^0,\la\in \La).$$
\end{defn}

 Following \cite{li18}, we consider only self-similar $k$-graphs $(G,\La,\varphi)$ for {\bf row-finite} and {\bf source-free} $k$-graphs with $|\La^0|<\infty$. We will write $(G,\La,\varphi)$ by $(G,\La)$ for simplicity. Note that $\varphi$ was called the restriction map in \cite{li18} and each $\varphi(g,\la)$ was denoted by $g|_\la$ there.

\begin{defn}
Let $(G,\La)$ be a self-similar $k$-graph. We say that
\begin{enumerate}
  \item $(G,\La)$ is {\it pseudo free}, if $g \la=\la$ and $\varphi(g,\la)=1_G$ imply $g=1_G$.
  \item $(G,\La)$ is {\it $G$-aperiodic} if for any $v\in \La^0$, there exists $x\in v\La^\infty$ such that $x(p,\infty)=g  x(q,\infty)$ implies $g=1_G$ and $p=q$ for $p,q\in\N^k$ and $g\in G$.
  \item $(G,\La)$ is {\it $G$-cofinal} if for every $x\in \La^\infty$ and $v\in\La^0$, there exist $p\in\N^k$, $\mu\in\La$, and $g\in G$ such that $s(\mu)=x(p,p)$ and $r(\mu)=g v$.
\end{enumerate}
\end{defn}

\begin{defn}
Let $(G,\La)$ be a self-similar $k$-graph as in Definition \ref{defn5.1} with $|\La^0|<\infty$. The $C^*$-algebra $\Ol$ associated to $(G,\La)$ is the universal $C^*$-algebra generated by $\{s_\la:\la\in \La\}$ and $\{u_g:g\in G\}$ such that
\begin{enumerate}
  \item $\{s_\la:\la\in \La\}$ is a Cuntz-Krieger $\La$-family in the sense of \cite{kum00}.
  \item $u:G\rightarrow \Ol$, given by $g\mapsto u_g$, is a unitary $*$-representation of $G$.
  \item $u_gs_\la=s_{g \la}u_{\varphi(g,\la)}$ for every $g\in G$ and $\la\in \La$.
\end{enumerate}
\end{defn}

Similar to the construction of $\g$ in Section 2.4, Li and Yang associated an ample groupoid $\mathcal{G}_{G,\La}$ in \cite[Section 5.1]{li18} such that $\mathcal{O}_{G,\La}\cong C^*(\mathcal{G}_{G,\La})\cong C_r^*(\mathcal{G}_{G,\La})$ when $G$ is amenable and $(G,\La)$ is pseudo free \cite[Theorem 5.9]{li18}. In particular, the unit space $\mathcal{G}_{G,\La}^{(0)}$ is homeomorphic to $\La^\infty$ endowed with the topology generated by cylinders $Z(\la):=\{\la x:x\in\La^\infty\}$.

Recall that a {\it circuit} in $\La$ is a path $\alpha\in \La$ with $r(\alpha)=s(\alpha)$. $\tau\in \La$ is called an {\it entry} for $\alpha$ if $r(\tau)=r(\alpha)$ and there are no common extensions for $\alpha$ and $\tau$ (i.e., $\alpha\mu\neq\tau\nu$ for all $\mu,\nu\in\La$).

\begin{thm}
Let $(G,\Lambda)$ be a pseudo free self-similar $k$-graph with $|\Lambda^0|<\infty$ over an amenable group $G$. If $\Lambda$ is $G$-aperiodic, then $\mathcal{O}_{G,\Lambda}$ is purely infinite. In particular, if $\Lambda$ is also $G$-cofinal, then $\mathcal{O}_{G,\Lambda}$ is a Kirchberg algebra.
\end{thm}

\begin{proof}
Let $\mathcal{G}_{G,\Lambda}$ be the groupoid associated to $(G,\Lambda)$. Then $\mathcal{G}_{G,\Lambda}$ is amenable and effective \cite[Proposition 6.5]{li18}, and we thus have $C^*(\mathcal{G}_{G,\Lambda})=C^*_r(\mathcal{G}_{G,\Lambda})=\mathcal{O}_{G,\Lambda}$ by \cite[Theorem 5.9]{li18}. We know that the cylinders $\{Z(\la):\la\in \Lambda\}$ form a basis of compact open sets for the topology on $\Lambda^\infty=\mathcal{G}_{G,\Lambda}^{(0)}$. So, in light of Proposition \ref{prop2.3}, it suffices to prove that each $1_{Z(\la)}$ is an infinite projection for $\la\in \Lambda$. For this, since
$$1_{Z(\la)}=s_\la s_\la^*\sim s_\la^* s_\la=s_{s(\la)},$$
we show all $s_v$'s are infinite in $\mathcal{O}_{G,\Lambda}$ for $v\in \Lambda^0$.

So fix an arbitrary $v\in \Lambda^0$. We claim that $v$ reaches from a circuit with an entry. To see this, take some $x\in v\Lambda^\infty$. For any $t\in \mathbb{N}$, write $\textbf{t}:=(t,0,0,\ldots)\in \mathbb{N}^k$. Since $\{x(\textbf{t},\textbf{t}):t\geq 1\}\subseteq \Lambda^0$ is finite, there are $t_1<t_2$ such that $x(\textbf{t}_1,\textbf{t}_1)=x(\textbf{t}_2,\textbf{t}_2)$. Hence $x(\textbf{t}_1,\textbf{t}_2)$ is a circuit in $\Lambda$, which connects to $v$ by $x(0,\textbf{t}_1)\in\La$. Note that the $G$-aperiodicity yields clearly the periodicity of $\La$. Hence, one may follow \cite[Lemma 6.1]{lar18} to find an (initial) circuit $\alpha$ with an entry $\tau$ connecting to $v$, as claimed.

Since $\alpha$ and $\tau$ have no common extensions, one may compute that $s_\alpha s_\alpha^*$ and $s_\tau s_\tau^*$ are orthogonal (by applying \cite[Lemma 9.4]{rae05}). Thus, by the Cuntz-Krieger relations we have
$$s_{r(\alpha)}\geq s_\alpha s_\alpha^*+ s_\tau s_\tau^*>s_\alpha s_\alpha^*\sim s_\alpha^* s_\alpha=s_{s(\alpha)}=s_{r(\alpha)},$$
so $s_{r(\alpha)}$ is infinite. Moreover, if $\lambda$ connects $r(\alpha)$ to $v$, then
$$s_v\geq s_\lambda s_\lambda^*\sim s_\lambda^*s_\lambda=s_{s(\lambda)}=s_{r(\alpha)},$$
which says that $s_v$ is an infinite projection in $\mathcal{O}_{G,\Lambda}$ as well. Since $v\in \Lambda^0$ was arbitrary, this deduces that $\mathcal{O}_{G,\Lambda}$ is purely infinite by Proposition \ref{prop2.3}.

For the last statement, if moreover $\Lambda$ is $G$-cofinal, then \cite[Theorem 6.6]{li18} implies that $\mathcal{O}_{G,\Lambda}$ is nuclear and simple, which satisfies UCT. Hence, $\mathcal{O}_{G,\Lambda}$ is a Kirchberg algebra.
\end{proof}

\begin{cor}[{See \cite[Theorem 6.13]{li18}}]
Let $(G,\Lambda)$ be a pseudo free self-similar $k$-graph with $|\Lambda^0|<\infty$ over an amenable group $G$. Whenever $\mathcal{O}_{G,\Lambda}$ is simple, then it is purely infinite too.
\end{cor}



\begin{thebibliography}{99}


\bibitem{ana00}
C. Anantharaman-Delaroche and J. Renault. \emph{Amenable groupoids}, volume 36 of Monographs of L'Enseignement Mathématique, Geneva, 2000.

\bibitem{bed17}
E. B\'{e}dos, S. Kaliszewski, J. Quigg, \emph{On Exel-Pardo algebras}, J. Operator Theory {\bf 78}(2) (2017), 309-345.

\bibitem{bon18}
C. B\"{o}nicke and K. Li, \emph{Ideal structure and pure infiniteness of ample groupoid $C^\ast$-algebras}, Ergodic Theory Dynam. Systems (2018), 1-30. doi:10.1017/etds.2018.39

\bibitem{bro15}
J.H. Brown, L.O. Clark and A. Sierakowski, \emph{Purely infinite $C^\ast$-algebras associated to \'{e}tale groupoids}, Ergodic Theory Dynam. Systems {\bf35} (2015), 2397-2411.

\bibitem{cun81}
J. Cuntz. \emph{$K$-theory for certain $C^*$-algebras}, Ann. Math. {\bf113} (1981), 181-197.

\bibitem{exe10}
R. Exel, \emph{Non-Hausdorff $\acute{e}$tale groupoids}, Proc. Amer. Math. Soc. {\bf139} (2011), 897-907.

\bibitem{exe17}
R. Exel, E. Pardo, \emph{Self-similar graphs, a unified treatment of Katsura and Nekrashevych $C^*$-algebras}, Adv. Math. {\bf 306} (2017), 1046-1129.

\bibitem{exe18}
R. Exel, E. Pardo and C. Starling, \emph{$C^*$-algebras of self-similar graphs over arbitrary graphs}, preprint, arXiv:1807.01686 (2018).

\bibitem{hje01}
J. Hjelmborg, \emph{Purely infinite and stable $C^*$-algebras of graphs and dynamical systems}, Ergodic Theory Dynam. Systems {\bf21} (2001), 1789-1808.

\bibitem{kat08}
T. Katsura, \emph{A construction of actions on Kirchberg algebras which induce given actions on their $K$-groups}, J. Reine Angew. Math. {\bf617} (2008), 27-65.

\bibitem{kirch00}
E. Kirchberg and M. R{\o}rdam, \emph{Non-simple purely infinite $C^*$-algebras}, Amer. J. Math. {\bf122} (2000), 637-666.

\bibitem{kum00}
A. Kumjian and D. Pask, \emph{Higher rank graph $C^*$-algebras}, New York J. Math. {\bf 6} (2000), 1-20.

\bibitem{kum98}
A. Kumjian, D. Pask and I. Raeburn, \emph{Cuntz-Krieger algebras of directed graphs}, Pacific J. Math. {\bf184} (1998), 161-174.

\bibitem{lal19}
S.M. LaLonde, D. Milan, and J. Scott, \emph{Condition (K) for inverse semigroups and the ideal structure of their $C^*$-algebras}, J. Algebra {\bf 523} (2019), 119-153.

\bibitem{lar19}
H. Larki, \emph{Non-simple purely infinite Steinberg Algebras with applications to Kumjian-Pask algebras}, preprint, arXiv:1901.07094 (2019)

\bibitem{lar18}
H. Larki, \emph{Purely infinite simple Kumjian-Pask algebras}, Forum Math. {\bf30}(1) (2018), 253-268.

\bibitem{li18}
H. Li and D. Yang, \emph{Self-similar $k$-graph $C^*$-algebras}, preprint, arXiv:1712.08194 (2018).

\bibitem{nek04}
V. Nekrashevych, \emph{Cuntz-Pimsner algebras of group actions}, J. Operator Theory {\bf 52} (2004), 223-249.

\bibitem{nek05}
V. Nekrashevych, Self-Similar Groups, Mathematical Surveys and Monographs, vol.117, Amer. Math. Soc., Providence RI 2005.

\bibitem{rae05}
I. Raeburn, Graph Algebras, CBMS Regional Conf. Ser. in Math., vol. 103, Amer. Math. Soc., Providence RI 2005.

\bibitem{rai18}
T. Rainone and A. Sims, \emph{A dichotomy for groupoid $C^*$-algebras}, Ergod. Th. Dynam. Sys. (2018), 1-43, doi:10.1017/etds.2018.52.

\bibitem{ren08}
J. Renault, \emph{Cartan subalgebras in $C^\ast$-algebras}, Irish Math. Soc. Bulletin {\bf61} (2008), 29–63.

\bibitem{ren80}
J. Renault, A groupoid approach to $C^\ast$-algebras, Lecture Notes in Mathematics, vol. 793, Springer, Berlin, 1980.

\bibitem{ror00}
M. R${\o}$rdam, F. Larsen and N. Laustsen, An introduction to $K$-theory for $C^*$-algebras, London Mathematical Society Student Texts 49, Cambridge University Press, Cambridge, 2000.

\bibitem{ror02}
M. R${\o}$rdam and E. St${\o}$rmer, Classification of nuclear $C^*$-algebras Entropy in operator algebras,
Operator Algebras and Non-commutative Geometry, 7, Springer-Verlag, Berlin, 2002.

\end{thebibliography}
\end{document}